\documentclass{amsart}
\title[Residue formula and its applications]{A residue formula for meromorphic connections and applications to stable sets of foliations}


\date{November 4, 2022}

\author{Masanori Adachi}
\address[M. Adachi]{Department of Mathematics, Faculty of Science, Shizuoka University.  836 Ohya, Suruga-ku, Shizuoka 422-8529, Japan.}
\email{adachi.masanori@shizuoka.ac.jp}

\author{S\'everine Biard}
\address[S. Biard]{Univ. Polytechnique Hauts-de-France, INSA Hauts-de-France, CERAMATHS-Laboratoire de Mat\'eriaux C\'eramiques et de Math\'ematiques, F-59313 Valenciennes, France}

\email{severine.biard@uphf.fr}

\author{Judith Brinkschulte}
\address[J. Brinkschulte]{Universit\"at Leipzig, Mathematisches Institut, PF 100920, D-04009 Leipzig, Germany}
\email{brinkschulte@math.uni-leipzig.de}

\thanks{MA is partially supported by JSPS KAKENHI Grant Numbers JP18K13422, JP19KK0347, JP21H00980 and JP21K18579. This work was completed during his stay at Universit\"at zu K\"oln, to which he is grateful for the hospitality. SB and JB are partially supported by ANR-DFG project QuaSiDy - Quantization, Singularities and Holomorphic Dynamics. SB is also partially supported by FR 2037 CNRS}
\usepackage{amsmath,amsthm,amssymb,latexsym}
\usepackage[abbrev]{amsrefs}
\usepackage[all]{xy}
\usepackage{color}
\usepackage{enumitem}

\newtheorem{Theorem}{Theorem}[section]
\newtheorem*{ResidueTheorem}{Residue Formula}

\newtheorem{Claim}[Theorem]{Claim}

\newtheorem{Proposition}[Theorem]{Proposition}
\newtheorem{Definition}[Theorem]{Definition}
\newtheorem{Lemma}[Theorem]{Lemma}

\theoremstyle{remark}
\newtheorem{Remark}[Theorem]{Remark}

\newcommand\C{\mathbb{C}}  
\newcommand\R{\mathbb{R}}

\newcommand\Z{\mathbb{Z}}
\newcommand\D{\mathbb{D}}

\newcommand\Ker{\mathrm{Ker}}
\renewcommand\Re{\operatorname{Re}}
\renewcommand\Im{\operatorname{Im}}

\newcommand\Sing{\operatorname{Sing}}
\newcommand{\pa}{\partial}
\newcommand{\opa}{\overline\pa}
\newcommand{\ol}{\overline }

\begin{document}

\begin{abstract}
We discuss residue formulae that localize the first Chern class of a line bundle to the singular locus of a given holomorphic connection. As an application, we explain a proof for Brunella's conjecture about exceptional minimal sets of codimension one holomorphic foliations with ample normal bundle and for a nonexistence theorem of Levi flat hypersurfaces with transversely affine Levi foliation in compact K\"ahler surfaces.
\end{abstract}

\maketitle

\section{Introduction}
It is a well-known fact that for a logarithmic or an integrable meromorphic connection of a holomorphic line bundle, one can define its residues (cf. \cites{deligne,Saito} for logarithmic connections). 
Residue formulae computing the first Chern class of the line bundle in terms of the residues were shown for logarithmic connections on compact complex manifolds in \cite{ohtsuki} and for integrable meromorphic connections on projective manifolds in \cite{cousin-pereira}*{Proposition 2.2}. 
Their extensions to arbitrary complex manifolds were provided recently by Pereira \cite{Pereira}*{Propositions 3.3 and 3.4},
to which we refer the reader for the details. 

The purpose of this paper is to discuss applications of residue formulae to the study of stable sets of holomorphic foliations of codimension one. 
In particular, we shall make use of the following residue formula for holomorphic connections which are not necessarily integrable a priori but whose curvature forms holomorphically extend across divisors in complex manifolds: 
\begin{ResidueTheorem}
Let $X$ be a compact K\"ahler manifold of dimension $n \geq 1$, $L \to X$ a holomorphic line bundle, and $D$ a reduced divisor.
We set $X' := X \setminus D$.
Suppose that we are given a holomorphic connection $\nabla$ of $L$ over $X'$ such that its curvature form holomorphically extends over $X$. Then the connection $\nabla$ is integrable and the first Chern class of $L$ decomposes as
\[
c_1(L) = - c(\operatorname{Res}(\nabla)) \in H^{1,1}(X)
\]
where $c(\operatorname{Res}(\nabla))$ denotes the residue divisor of $\nabla$ $($see Section \ref{sect:residue} for its definition$)$.
\end{ResidueTheorem}
The notion of residue and the residue formula in this setting seem to be folklore among experts including Brunella (cf. \cite{Pereira-Pirio}*{Proof of Proposition 7.1}). 
In Section \ref{sect:residue}, we offer their proofs (Lemma \ref{lem:residue-existence} and Theorem \ref{thm:residue}) following the recent survey by Pereira \cite{Pereira}.

In order to apply this residue formula, we have to know when the curvature form of a holomorphic connection of a line bundle extends across the divisor. 
When $\dim X = n \geq 4$ and the divisor has a strongly pseudoconvex neighborhood, the curvature form, being a holomorphic form  of degree $\leq n-2$, always extends due to Ohsawa \cite{ohsawa}*{Corollary 7}. In dimension 3, we cannot rely on Ohsawa's result. To this extent we prove that a $d$-closed holomorphic $(n-1)$-form holomorphically extends across the maximal compact analytic subset of a strongly pseudoconvex  manifold of dimension $n$
 (Theorem \ref{thm:extension}).
 We thus obtain a special case of the result in \cite{vSS}*{Corollary 1.4} by using
 $L^2$ theory for the $\opa$-operator on complete K\"ahler manifolds as in \cite{ohsawa}
  instead
of  local algebraic Hodge theory.\\

We then proceed to discuss applications of these results to holomorphic foliations and Levi flat hypersurfaces.
A smooth real hypersurface $M$ in a complex manifold of dimension $n\geq 2$ is called \textit{Levi flat} if its Levi form vanishes identically or, equivalently, if it admits a smooth foliation by complex hypersurfaces called the \textit{Levi foliation} of $M$. When $M$ is real analytic, the Levi foliation can be extended to a neighborhood of $M$ as a holomorphic foliation of codimension one. 
Even though Levi flat hypersurfaces appeared early in counter-examples of pseudoconvex domains that are not Stein, the interest in them rose in connection with the following fundamental question in foliation theory: \textit{Can a leaf of a holomorphic foliation $\mathcal{F}$ of codimension one in $\mathbb{CP}^n$ accumulate to the singular set $\Sing(\mathcal{F})$?}
If not, there is a (non-empty) compact set, invariant by $\mathcal{F}$, minimal with respect to the inclusion, disjoint from $\text{Sing}(\mathcal{F})$ that is called an \textit{exceptional minimal set}. By Cerveau's dichotomy result \cite{Cerveau93}, this set has to be a real analytic Levi flat hypersurface or has linearizable abelian holonomy. The non-existence of such a set has been proved in $\mathbb{CP}^n$, $n\geq 3$, by Lins Neto in \cite{linsneto}, but the case of $\mathbb{CP}^2$ is still open.\\

\indent From the attempt to prove the conjecture about non-existence of real analytic Levi flat hypersurfaces in $\mathbb{CP}^2$ stems an interest in classifying real analytic Levi flat hypersurfaces in compact surfaces in terms of the geometry of its complement. 
For example, some natural and simple foliations such as  transversely affine holomorphic foliations, whose change of coordinates of the foliated atlas are affine in the normal direction, carry an interesting structure, given by a couple of $d$-closed 1-forms. 
This structure induces an integrable (flat) connection on the normal bundle $N_\mathcal{F}$ to the foliation $\mathcal{F}$. 
%When this bundle is a holomorphic line bundle away from some simple normal crossing divisor $D$, the residue of the connection on $N_\mathcal{F}$ to each irreducible component of the divisor is well defined. 
In \cite{cousin-pereira}, the authors localize the first Chern class of the normal bundle in terms of the residues of this connection when the manifold is projective. This localization allows in \cite{canales} to deduce a dynamical property: a real analytic Levi flat hypersurface in an algebraic surface whose Levi foliation is transversely affine  has an invariant transverse measure. Such Levi flat hypersurfaces contain either a compact leaf or are defined by a closed one-form.  Using a residue formula that localizes the first Chern class to the singular locus of a logarithmic connection, obtained by \cite{ohtsuki}, the authors in \cite{adachi-biard} prove the nonexistence of  real analytic closed Levi flat hypersurfaces in compact K\"ahler surfaces whose Levi foliation is transversely affine and whose complement is 1-convex. \\

\indent In parallel, answering whether or not every leaf of a holomorphic foliation in a compact complex manifold accumulates to $\Sing(\mathcal{F})$ might be seen as a further generalization of nonexistence theorems of compact Levi flat hypersurfaces in dimension $\geq 3$. In \cite{brunella1}, Brunella generalized the question by conjecturing that every leaf of a holomorphic codimension one foliation whose normal bundle is ample accumulates to $\Sing(\mathcal{F})$ in compact complex manifolds of dimension $\geq 3$. The normal bundle to the  foliation in $\mathbb{CP}^n$ being automatically positive, Brunella's conjecture is a simple generalization of what happens in $\mathbb{CP}^n$. Brunella's conjecture was shown to be true recently in \cite{adachi-brinkschulte}.
In \cite{adachi-brinkschulte}, the authors used  Baum--Bott theory to localize the square of the first Chern class to the singular locus and then the first Atiyah form to localize the first Chern class in order to bypass the lack of residue formula for a holomorphic connection.\\

The residue formula stated above and Theorem \ref{thm:extension} allow us to give a different proof of Brunella's conjecture in \cite{adachi-brinkschulte}. 
Also a residue formula recently given in \cite{Pereira} (see Theorem \ref{thm:residue_integrable}) simplifies the proof of the non-existence of real analytic Levi flat  with transversely affine Levi foliation in \cite{adachi-biard}.
We discuss these applications of residue formulae in Section \ref{sect:applications}.

\subsection*{Acknowledgements}
The authors are grateful to Jorge Vit\'orio Pereira for suggesting the extension argument for foliations over exceptional sets, which was crucially used in the proof of Theorem \ref{thm:affine}, and for pointing out his recent survey \cite{Pereira}, which helps to improve some results in Section \ref{sect:residue}.
We are also grateful to Stefan Nemirovski for his helpful comments, in particular for pointing out the reference \cite{vSS}.

\section{Preliminaries}
In this section, we briefly recall basic notions and results that will be used in the sequel.

\subsection{Holomorphic connections}
Let us recall notions around holomorphic connections of holomorphic vector bundles over complex manifolds. 
We restrict ourselves to the case of line bundles, and refer the reader to \cites{deligne, ohtsuki} for backgrounds and details.

Let $X$ be a complex manifold and $L$ a holomorphic line bundle over $X$. We
denote by $\mathcal{O}_X(L)$ the sheaf of germs of holomorphic sections of $L$
and by $\Omega^1_X$ the sheaf of germs of holomorphic $1$-forms on $X$. 
A \emph{holomorphic connection} $\nabla$ on $L$ is a $\mathbb{C}$-linear sheaf morphism
$\nabla : \mathcal{O}_X(L) \longrightarrow \Omega^1_X\otimes_{\mathcal{O}_X}\mathcal{O}_X(L)$
that satisfies the Leibniz rule $\nabla(f s) = df \otimes s + f \nabla s$ for any holomorphic function
$f$ and holomorphic section $s$ of $L$ defined on a common open set.

In a local holomorphic trivialization of $L$, $\nabla$ is defined by a holomorphic 1-form $\eta$, which is called a \emph{connection form}.
The \emph{curvature} $\Theta$ of $\nabla$ is a holomorphic 2-form such that for every locally defined holomorphic section $s$ of $L$, $\nabla^2 s = \Theta \otimes s$, where $\nabla$ is extended on $L$-valued
holomorphic forms by the Leibniz rule. Locally, $\Theta = d\eta$. 
An \emph{integrable} or \emph{flat} connection $\nabla$ is one whose curvature vanishes identically, i.e.; $\Theta = 0$.

\subsection{Holomorphic foliations} 
Let $X$ be a complex manifold of dimension $n \geq 2$. 
In this paper, we discuss foliations in the following sense: 
\begin{Definition}
	\label{def:foliation}
	We say that a collection of holomorphic 1-forms $\mathcal{F} = \{ \omega_\mu \}$, where $\omega_\mu \in \Omega^1_X(U_\mu)$ and $\mathcal{U} = \{ U_\mu \}$ is an open covering of $X$, define 
	a \emph{codimension one holomorphic foliation} on $X$ if they satisfy the following conditions: for any $\mu$ and $\nu$, 
	\begin{enumerate}
		\item There exists $g_{\mu\nu} \in \mathcal{O}_X^*(U_\mu \cap U_\nu)$ such that $\omega_\mu = g_{\mu\nu} \omega_\nu$ on $U_\mu \cap U_\nu$;
		\item The analytic set $\{ p \in U_\mu \mid \omega_\mu(p) = 0\}$ has codimension $\geq 2$;
		\item The integrability condition is fulfilled: $\omega_\mu \wedge d\omega_\mu = 0$ on $U_\mu$.
	\end{enumerate}
\end{Definition}
\bigskip
The cocycle $\{ g_{\mu\nu} \}$ defines a holomorphic line bundle over $X$ called the \emph{normal bundle} of $\mathcal{F}$ and denoted  by  $N_{\mathcal F}$.
The dual bundle of $N_{\mathcal F}$ is called the \emph{conormal bundle} and denoted by $N^*_{\mathcal F}$. Note that $\{\omega_\mu\}$ defines a global $1$-form with values in $N_{\mathcal F}$.

From the first and second condition, the zero sets of the $\omega_\mu$'s glue together and define an analytic set of codimension $\geq 2$ on $X$. It is called the \emph{singular set} of $\mathcal{F}$ and denoted by $\Sing(\mathcal{F})$. 

Away from the singular set, the integrability condition implies that the kernels of the $\omega_\mu$'s define a holomorphic subbundle $T_\mathcal{F}$ of the holomorphic tangent bundle $T^{1,0}_{X}$ of corank one. 
It is integrable in the sense of Frobenius, and its maximal integral submanifolds are called the \emph{leaves} of $\mathcal{F}$. 
The normal bundle identifies then with the quotient bundle $T^{1,0}_X/ T_\mathcal{F}$.

Note that we can define a foliation from a non-trivial integrable meromorphic 1-form $\omega$ on $X$: by taking an open covering $\{ U_\mu \}$, we can find meromorphic functions $f_\mu$ on $U_\mu$ such that ${\omega}_\mu=f_\mu \omega$ is integrable holomorphic 1-form with a singular set of codimension at least $2$, which defines a foliation.

\subsection{Simple normal crossing divisors and modifications} \label{sect:mod}
 A divisor $D=\sum_\nu D_\nu$ in a complex manifold $X$ of dimension $n$ is a \emph{simple normal crossing divisor} if each irreducible component $D_\nu$ is smooth and intersects transversely. Locally, at a point $p$, there exist complex coordinates $(z_1,\dots,z_n)$ such that  $D=\lbrace{z\in X\mid z_1 z_2\dots z_k=0}\rbrace$ for some $0\leq k\leq n$. In other words, such divisor can be locally seen as a union of coordinates hyperplanes and then has locally a simple structure, that we will use in this paper. For the sake of notations, we identify a reduced divisor with its support. 

A simple normal crossing divisor naturally arises as the exceptional set of a strongly pseudoconvex manifold up to a \emph{proper modification}, a proper holomorphic map biholomorphic outside  nowhere dense closed analytic subsets.
 We recall below two modifications repeatedly used in the paper.

\indent The Remmert reduction allows to see a strongly pseudoconvex manifold as a proper modification of a Stein space.
An open complex manifold $X$ is said to be \emph{strongly pseudoconvex} or \emph{1-convex} when it admits a smooth plurisubharmonic exhaustion function which is strictly plurisubharmonic outside a compact subset. 
From a classical result of Grauert, there is a \emph{maximal} compact analytic set $A \subset X$ that contains all the compact analytic subsets in $X$ of positive dimension.
The \emph{Remmert reduction} of $X$ contracts $A$ to a finite set of points: 
there are a proper holomorphic map $\phi \colon X \to Y$ to a normal Stein space $Y$ and a finite set $B\subset Y$ such that $\phi \colon X\setminus A \rightarrow Y\setminus B$ is a biholomorphism (see for example \cite{Peternell}). 
In normal Stein spaces, we can use extension theorems such as Levi's extension theorem or a version of Bochner--Hartogs Theorem on analytic objects such as connections. 

\emph{Hironaka's desingularization} is a powerful modification to simplify the structure of the singularity. Let $X$ be a complex manifold and $B\subset X$ a compact analytic subset. There exists a
complex manifold $\tilde{X}$ and a proper holomorphic map $\pi\colon \tilde{X}\rightarrow X$ and a simple normal crossing divisor $D$ on $\tilde{X}$ such that $\pi|_{\tilde{X}\setminus D} \colon\tilde{X}\setminus D\rightarrow X\setminus B$ is a biholomorphism. 
Note that $\tilde{X}$ remains projective or K\"ahler if $X$ is projective or K\"ahler respectively. 

Those modifications intervene in different sections, to extend analytic objects such as a holomorphic connection or to get a simple normal crossing divisor.
We often identify objects on $X \setminus A$ with $Y \setminus B$ or $\tilde{X} \setminus D$ via the modification maps if there is no confusion.

Later we will pull back a foliation by a modification. Let $\phi \colon X\to Y$ be a proper modification between complex manifolds, i.e.; $\phi$ is proper and there exists a nowhere dense closed analytic subset $B$ of $Y$ such that $\phi: X\setminus \phi^{-1}(B)\to Y\setminus B$ is an isomorphism.
Suppose a codimension one holomorphic foliation is defined by a global integrable meromorphic 1-form $\omega$ on $Y$. Then we can define the pull-back foliation by the pull-back form $\phi^*\omega$, which is again an integrable meromorphic 1-form on $X$.
Since the integrability of $\phi^*\omega$ is clear, we check that $\phi^*\omega$ is meromorphic. Assuming, for convenience, that $X$ and $Y$ are surfaces and let $\omega_1$ and $\omega_2$ be global sections of $\Omega^1_Y$ such that $\omega_1\wedge\omega_2\not\equiv 0$. On a small open set $U\subset Y$, a meromorphic 1-form can be written $f_1 \omega_1 + f_2 \omega_2$ for a meromorphic function $f_1$ and $f_2$ on $U$. Hence,
$\phi^*(f_1 \omega_1 + f_2 \omega_2)= (\phi^*f_1) \phi^*\omega_1 + (\phi^*f_2) \phi^*\omega_2$. It is therefore enough to verify that $\phi^*f$ is meromorphic when $f$ is a meromorphic function. Indeed,
on a small $U$, let $g,h\in\mathcal{O}_Y(U)$ such that $f={g}/{h}$. Then $\phi^*f={g\circ \phi}/{h\circ \phi}$ is a meromorphic function on $\phi^{-1}(U)\subset X$ where $h\circ \phi$ cannot identically vanish on any open subset $V$ of $\phi^{-1}(U)$. Otherwise, $h$ would vanish on $\phi(V\setminus  \phi^{-1}(B))\subset U\setminus B$, that is impossible by definition of $\phi$.

\section{Residue formulae}
\label{sect:residue}
In this section, we first recall a residue formula for integrable meromorphic connections on line bundles based on recent survey by Pereira \cite{Pereira}. 
We will state it in a slightly more general setting, namely, for holomorphic connections defined away from reduced divisors, not assuming their singularities are polar. 
Then, following Brunella, we observe that we can replace the flatness assumption of the given connection with holomorphic extendability of its curvature form when the ambient manifold is compact K\"ahler.

Let $X$ be a complex manifold of dimension $n$, $L$ a holomorphic line bundle over $X$, and $D$ a reduced divisor with irreducible decomposition $D = \sum_{\nu = 1}^N D_\nu$.
We write $X' := X \setminus D$ and let $\nabla$ be a holomorphic connection of $L$ over $X'$.
The unit disk in $\C$ is denoted by $\D$.

We first assume that $\nabla$ is integrable. Then, for each irreducible component $D_\nu$, there is a well-defined number $\operatorname{Res}_{D_\nu}(\nabla) \in \C$, 
which we call the \emph{residue} of $\nabla$ along $D_\nu$, that satisfies the following property: 
for any smooth point $p \in D_\nu$ of $D$, there exists a neighborhood $U$ of $p$ in $X$
such that for any holomorphic embedding $\iota \colon \overline{\D} \to U$ intersecting with $D_\nu$ at $p$ transversely, we have
\[
\operatorname{Res}_{D_\nu}(\nabla) = \frac{1}{2\pi i}\int_{\pa\D} \iota^*\eta,
\]
where $\eta$ is the connection 1-form of $\nabla$ with respect to arbitrary trivialization of $L$ over $U$.
We define the \emph{residue divisor} of $\nabla$ by
\[
\operatorname{Res}(\nabla) := \sum_{\nu = 1}^N \operatorname{Res}_{D_\nu}(\nabla) D_\nu
\]
as a divisor with $\C$-coefficients. 
Each divisor $D$ with $\C$-coefficients naturally defines a singular cohomology class in $H^2(X, \C)$. 
We call it the \emph{Chern class} of $D$ denoted by $c(D)$.

By adapting the proofs for \cite{Pereira}*{Theorems 3.1 and 3.3} (cf. Weil \cite{Weil}), we have the following residue formula.
\begin{Theorem}[cf. Pereira \cite{Pereira}*{Theorem 3.3}]
\label{thm:residue_integrable}
Let $X$ be complex manifold, which is not necessarily compact, $L \to X$ a holomorphic line bundle, and $D$ a reduced divisor.
We set $X' := X \setminus D$.
Suppose that we are given an integrable holomorphic connection $\nabla$ of $L$ over $X'$. Then the first Chern class of $L$ decomposes as
\[
c_1(L) = - c(\operatorname{Res}(\nabla)) \in H^{2}(X, \C).
\]
\end{Theorem}

Notice that the meromorphicity of the connection was not essentially used in the proof of \cite{Pereira}*{Theorem 3.3}.

Next we discuss a residue formula for holomorphic connections with possibly non-vanishing curvature. 
To assure the well-definedness of residues, we need some control on the singularity of our connection $\nabla$ along the divisor $D$.
We assume that the curvature form of $\nabla$ extends holomorphically across $D$. 

\begin{Lemma}
\label{lem:residue-existence}
Let $\nabla$ be a holomorphic connection of $L$ over $X'$, which is not necessarily integrable. 
Assume that the curvature 2-form $\Omega$ of $\nabla$ extends holomorphically over $X$.
Then the residue $\operatorname{Res}_{D_\nu}(\nabla) \in \C$ is well-defined. 
\end{Lemma}

Although this fact seems to be well-known to experts (cf. \cite{Pereira-Pirio}*{Proof of Proposition 7.1}), we give a proof for the readers' convenience. 

\begin{proof}
Take a smooth point $p \in D_\nu$ of $D$
and a holomorphic chart $(U; z_1, z_2, \dots, z_n)$ trivializing $L$ such that $U \simeq \D^n$ and $U \cap D_\nu \simeq \{0\} \times \D^{n-1}$.
We will use the notation $z' = (z_2, \cdots, z_n)$.
Fix a local trivialization of $L$ over $U$ and let $\eta$ be the connection 1-form of $\nabla$ with respect to this trivialization, which is a holomorphic 1-form on $U \setminus D_\nu$.

We expand $\eta$ into a Laurent series in $z_1$ on $(\D \setminus \{0\}) \times \D^{n-1}$: 
\[
\eta  = \sum_{k \in \Z} z_1^k \eta_k, \quad \eta_k = f_k(z') dz_1 + \eta_k', \quad \eta_k' = \sum_{j=2}^n g_{k,j}(z') dz_j
\]
where $f_k, g_{k,j} \in \mathcal{O}(\D^{n-1})$. We decompose $\eta$ into 
\begin{equation} \label{eq:decomposition1}
\eta = \eta_+ + \eta_-, \quad \eta_{+}  = \sum_{k \geq 0} z_1^k \eta_k, \quad \eta_{-}  = \sum_{k < 0} z_1^k \eta_k.
\end{equation}
Since $\eta_{+}$ extends as a holomorphic 1-form to $U$, we get
$\int_{\pa\D} \iota^*\eta_{+} = 0$ from Cauchy's integration theorem. 
We now compute $\int_{\pa\D} \iota^*\eta_{-}$.

The assumption that the curvature 2-form $\Omega = d\eta$ extends holomorphically across $D$
implies that $d\eta_- = 0$, that is, 
\begin{align*}
0 = d \eta_{-}  
&= \sum_{k < 0} (k z_1^{k-1} dz_1 \wedge \eta_k + z_1^{k} d\eta_k)\\
&=  \sum_{k < 0} z_1^k ((k+1) dz_1 \wedge \eta_{k+1} + d\eta_k).
\end{align*}
Hence, for $k < 0$,
\[
0 = (k+1) dz_1 \wedge \eta_{k+1} + d\eta_k
= dz_1 \wedge \left((k+1) \eta'_{k+1} - df_k\right) + d\eta_k'
\]
and, therefore, $df_k = (k+1) \eta'_{k+1}$.
In particular, when $k = -1$, we obtain $df_{-1} = 0$ and we see that $f_{-1}$ is a constant function on $\D^{n-1}$. 
It follows that
\begin{equation} \label{eq:decomposition2}
\begin{split}
 \eta_{-}  
& = \sum_{k < 0} z_1^k (f_k dz_1 + \eta_k')\\
& = f_{-1} \frac{dz_1}{z_1} + \sum_{k < 0} z_1^k (\frac{f_{k-1}}{z_1} dz_1 + \frac{df_{k-1}}{k})\\
& = f_{-1} \frac{dz_1}{z_1} + d \left(\sum_{k < 0} \frac{z_1^k f_{k-1}}{k}\right).
\end{split}
\end{equation}
We therefore have
\[
\int_{\pa\D} \iota^*\eta =\int_{\pa\D} \iota^*\eta_- = f_{-1} \int_{\pa\D}\frac{dz_1}{z_1} = 2\pi i f_{-1}.
\]

Since $\int_{\pa\D} \iota^*\eta$ does not depend on the choice of the holomorphic coordinate on $U$ 
nor on the trivialization of $L$ over $U$, $f_{-1}$ induces a well-defined locally constant function on the smooth points of $D$ on $D_\nu$. 
Hence,  $\operatorname{Res}_{D_\nu}(\nabla) := f_{-1}$ is well-defined.
\end{proof}

By adding a curvature term to the formula in Theorem \ref{thm:residue_integrable}, a residue formula holds for a holomorphic connection whose curvature form extends over entire $X$.

\begin{Proposition}
\label{prop:residue_with_curvature}
Let $X$ be a complex manifold, $\pi \colon L \to X$ a holomorphic line bundle, and $D$ a reduced divisor.
We set $X' := X \setminus D$.
Suppose that we are given a holomorphic connection $\nabla$ of $L$ over $X'$ whose curvature form $\Omega$ extends holomorphically on $X$. Then the first Chern class of $L$ decomposes as
\[
c_1(L) = - c(\operatorname{Res}(\nabla)) + \frac{i}{2\pi}[\Omega] \in H^{2}(X, \C).
\]
\end{Proposition}

Although this essentially follows from the proofs of \cite{Pereira}*{Theorems 3.1 and 3.3}, we reproduce the proofs here for the readers' convenience. We first consider the case when $L$ is the trivial line bundle.

\begin{Lemma}
\label{lem:weil}
Let $X$ be a complex manifold and $D$ a reduced divisor with irreducible decomposition $D = \sum_\nu D_\nu$.
Let $\eta$ be a holomorphic 1-form defined on $X \setminus D$, regarded as a holomorphic connection of the trivial line bundle, such that $d\eta$ extends holomorphically over $X$. Then
\[
c(\operatorname{Res}(\eta)) = \frac{i}{2\pi}[d\eta] \in H^2(X,\C).
\]
\end{Lemma}

\begin{proof}(following the proof of \cite{Pereira}*{Theorem 3.1})
Let $\sigma \in H_2(X, \Z)$ be an arbitrary homology class. 
We compute the pairing $c(\operatorname{Res}(\eta)) \cdot \sigma$.
Denote by $\Delta^2 \subset \R^2$ the standard 2-simplex. 
Represent $\sigma$ by a singular 2-cycle $\sum_j \alpha_j \sigma_j$, $\alpha_j \in \Z$,
such that the image of each 2-simplex $\sigma_j \colon \Delta^2 \to X$ either does not intersect $D$ or intersect $D$ transversely at a unique point that lies on the smooth locus of $D$. 
We denote 
\[
J_0 := \{ j \mid \sigma_j \cap D = \varnothing \}, \quad J_1 := \{ j \mid D \cdot \sigma_j = 1\}.
\]
Furthermore we may choose $\sigma$ so that the images of the boundaries of the 2-simplexes $\sigma_j$ do not intersect $D$. 

For each $\varepsilon > 0$, we take a subdivision $\{ \sigma^\varepsilon_{j,k} \}_{k=0}^{N^\varepsilon_j}$ of $\sigma_j$ for $j \in J_1$ so that the diameter of the unique face $\sigma^\varepsilon_{j,0}$ intersecting $D$ tends to 0 as $\varepsilon \searrow 0$. 
For small enough $\varepsilon$, we can work with a chart $U$ as in the proof of Lemma \ref{lem:residue-existence} and obtain the same decomposition \eqref{eq:decomposition1} and \eqref{eq:decomposition2} for $\eta$. It follows that
\[
\lim_{\varepsilon \searrow 0} \frac{1}{2\pi i}\int_{\partial \Delta^2} (\sigma^{\varepsilon}_{j,0})^*\eta
= \operatorname{Res}_{D_\nu}(\eta)
\]
where $D_\eta$ is the component that intersects with $\sigma_j$.
Therefore we can write
\begin{align*}
c(\operatorname{Res}(\eta)) \cdot \sigma 
&= \sum_{j \in J_1} \alpha_j \left(\sum_\nu \operatorname{Res}_{D_\nu}(\eta) D_\nu \cdot \sigma_j\right) = \lim_{\varepsilon \searrow 0}\frac{1}{2\pi i} \sum_{j \in J_1} \alpha_j \int_{\pa\Delta^2} (\sigma^\varepsilon_{j,0})^*\eta
\end{align*}
Since $\pa \sigma =0$, 
\[
\sum_{j \in J_1} \alpha_j \pa \sigma_j = -\sum_{j \in J_0} \alpha_j \pa \sigma_j
= -\sum_{j \in J_0} \alpha_j \sum_{k = 0}^{N^\varepsilon_j} \pa \sigma_{j,k}^{\varepsilon}
\]
as singular 1-chain. Hence, by Stokes' formula,
\begin{align*}
c(\operatorname{Res}(\eta)) \cdot \sigma 
&= - \lim_{\varepsilon \searrow 0} \frac{1}{2\pi i} \left(\sum_{j \in J_0} \alpha_j \int_{\pa\Delta^2} \sigma_j^*\eta + \sum_{j \in J_0} \alpha_j \sum_{k = 1}^{N^\varepsilon_j} \int_{\pa\Delta^2} (\sigma_{j,k}^{\varepsilon})^*\eta \right)\\
&= - \lim_{\varepsilon \searrow 0} \frac{1}{2\pi i} \left(\sum_{j \in J_0} \alpha_j \int_{\Delta^2} \sigma_j^*(d\eta) + \sum_{j \in J_0} \alpha_j \sum_{k = 1}^{N^\varepsilon_j} \int_{\Delta^2} (\sigma_{j,k}^{\varepsilon})^*(d\eta) \right)\\
&= -\frac{1}{2\pi i} \sum_{j} \alpha_j \int_{\Delta^2} \sigma_j^*d\eta = \frac{i}{2\pi}[d\eta] \cdot \sigma.
\end{align*}
\end{proof}

\begin{proof}[Proof for Proposition \ref{prop:residue_with_curvature}]
(following the proof of \cite{Pereira}*{Theorem 3.3}) On each local trivialization of $L$, 
let $\tilde{\eta} := d\zeta/\zeta + \eta$ where $\zeta$ denotes the local fiber coordinate and $\eta$ denotes the connection form of $L$ with respect to this local trivialization.
This glues together and defines a global holomorphic 1-form $\tilde{\eta}$ defined on $L \setminus (\pi^{-1}(D) \cup X)$ where $X$ is regarded as the zero section of $L$.

By construction, $\operatorname{Res}(\tilde{\eta}) = X + \pi^*(\operatorname{Res}(\nabla))$.
From Lemma \ref{lem:weil}, $c(\operatorname{Res}(\tilde{\eta})) =  i[\pi^*\Omega]/2\pi \in H^2(L, \C)$.
Since $c(X)|_X = c_1(L) \in H^2(X, \C)$, it follows that
\[
\frac{i}{2\pi}[\Omega] = c(\operatorname{Res}(\tilde{\eta}))|_X = c_1(L) + c(\operatorname{Res}(\nabla)) \in H^2(X, \C).
\]
\end{proof}

Brunella observed that, when $X$ is compact K\"ahler, the holomorphic extendability of the curvature form implies the integrability of the connection (cf. \cite{Pereira-Pirio}*{Proposition 7.1}). We therefore recover the same residue formula as Theorem \ref{thm:residue_integrable} in this case.

\begin{Theorem}
\label{thm:residue}
Let $X$ be a compact K\"ahler manifold, $L \to X$ a holomorphic line bundle, and $D$ a reduced divisor.
We set $X' := X \setminus D$.
Suppose that we are given a holomorphic connection $\nabla$ of $L$ over $X'$ such that its curvature form holomorphically extends over $X$. Then the connection $\nabla$ is integrable and the first Chern class of $L$ decomposes as
\[
c_1(L) = - c(\operatorname{Res}(\nabla)) \in H^{1,1}(X).
\]
\end{Theorem}
\begin{proof}
Since $X$ is compact K\"ahler, the following Hodge decomposition holds:
\[
H^2(X, \C) \simeq H^{2,0}(X) \oplus H^{1,1}(X) \oplus H^{0,2}(X).
\]
The formula in Proposition \ref{prop:residue_with_curvature} decomposes to $\Omega = 0$ and 
$c_1(L) = - c(\operatorname{Res}(\nabla))$.
\end{proof}

\section{Extension of closed holomorphic forms}
\label{sect:extension}

The aim of this section is to prove the following

\begin{Theorem}
\label{thm:extension}
	Let $X$ be a strongly pseudoconvex manifold of dimension $n \geq 2$ with maximal compact analytic subset $A$. Then every $d$-closed holomorphic $(n-1)$-form on $X\setminus A$ extends to a holomorphic $(n-1)$-form on $X$.
\end{Theorem}

We point out that Theorem \ref{thm:extension} was proved in \cite{vSS}*{Corollary 1.4} using local algebraic Hodge theory; here we provide a new proof that is based on $L^2$ theory for the $\opa$-operator on complete K\"ahler manifolds.\\

\begin{Remark}  \label{rem:extension}
Under the same hypothesis of Theorem \ref{thm:extension}, every holomorphic $p$-form on $X\setminus A$ (not necessarily $d$-closed) extends to a holomorphic $p$-form on $X$ if $0\leq p \leq n-2$. A proof based on $L^2$ methods can be found in \cite{ohsawa}*{Corollary 7}.
\end{Remark}

\subsection{Preliminaries on $L^2$ theory for $\opa$ on complete  manifolds}

Let $X$ be a complex manifold of dimension $n$ endowed with a smooth Hermitian metric $\omega$. 
We write $\mathcal{D}^{p,q}$ for the sheaf of germs of compactly supported smooth differential forms of type $(p,q)$.
By $L^2_{p,q}(X,\omega)$ (resp. $L^2_k (X,\omega)$) we denote the space of $(p,q)$-forms on $X$ (resp. $k$-forms on $X$) that are integrable with respect to the global $L^2$-norm
$$\Vert u\Vert^2 = \int_X \vert u(x)\vert^2_\omega dV_\omega,$$
where $\vert u(x)\vert_\omega$ is the pointwise Hermitian norm and $dV_\omega$ is the volume form.

On $X$ we consider the differential operators $d,\pa,\opa$, as well as the corresponding Laplace operators (depending on the metric $\omega$)
$$\Delta = d d^\ast + d^\ast d,\quad \Delta^\prime = \pa\pa^\ast + \pa^\ast\pa,\quad \Delta^{\prime\prime} = \opa\opa^\ast + \opa^\ast\opa .$$
All these operators extend naturally as linear, closed, densely defined operators on the previously defined $L^2$-spaces in the sense of distributions.
We will then consider the spaces of harmonic forms
$$\mathcal{H}^k(X,\omega)= \lbrace u\in L^2_k(X,\omega)\mid \Delta u =0\rbrace,$$
$$\mathcal{H}^{p,q}(X,\omega)= \lbrace u\in L^2_{p,q}(X,\omega)\mid \Delta^{\prime\prime} u =0\rbrace.$$

We will also need the following $L^2$ de Rham and Dolbeault cohomology groups
$$H^k_{L^2}(X,\omega) = L^2_k(X,\omega)\cap \Ker d / L^2_k(X,\omega)\cap \Im d,$$
$$H^{p,q}_{L^2}(X,\omega) = L^2_{p,q}(X,\omega)\cap \Ker \opa / L^2_{p,q}(X,\omega)\cap \Im \opa.$$
Then, if $\omega$ is complete, we have a natural isomorphism
$$\mathcal{H}^{p,q}(X,\omega) \simeq H^{p,q}_{L^2}(X,\omega),$$
provided the latter cohomology group is Hausdorff, as well as the corresponding isomorphism for the $L^2$ de Rham cohomology. 
In the following subsection, we will also need weighted $L^2$-spaces of the type $L^2_{p,q}(X,\omega,\varphi)$, consisting of forms that are integrable with respect to the weighted $L^2$-norm
$$\Vert u\Vert_\varphi^2 = \int_X \vert u(x)\vert^2_\omega e^{-\varphi}dV_\omega,$$
where $\varphi: X\longrightarrow\R$ is a (smooth) weight function. In this case,  the adjoint $\opa^\ast_\varphi$ depends not only on the metric $\omega$, but also on the weight function $\varphi$.

\bigskip
We refer the reader to \cite{demailly-book} for the details on the discussion above.

\subsection{Proof of the extension result}

\begin{proof}[Proof of Theorem \ref{thm:extension}] Using standard modifications recalled in Section \ref{sect:mod}, we denote by $E=\pi^{-1}(A)$ a divisor with simple normal crossings in a strongly pseudoconvex  manifold $\tilde X$ such that $\pi: \tilde X\longrightarrow X$ is a proper holomorphic mapping. %(Remmert reduction, Hironaka's desingularization) we can assume that we have a proper holomorphic map $\pi: \tilde X\longrightarrow X$, where $\tilde X$  is a strongly pseudoconvex K\"ahler manifold, $E= \pi^{-1}(A)$ is a divisor with simple normal crossings, and $\pi: \tilde X\setminus E\longrightarrow X\setminus A$ is biholomorphic.

Moreover $\tilde X$ can be chosen  such that there exist nonsingular divisors $E_1,\ldots, E_m$ and positive integers $p_1,\ldots,p_m$ such that $E = E_1\cup\ldots\cup E_m$ and $\otimes_{i=1}^m \mathcal{O}(-p_i E_i)$ is very ample on $\tilde X$. In particular this means that $\tilde{X}$ is K\"ahler.\\

We will first show that  any $d$-closed holomorphic $(n-1)$-form $\Omega$ on $\tilde X\setminus E$, there exists a holomorphic extension of $\Omega$ to $\tilde X$.\\

To this extent we let $W$ be a sufficiently small  neighborhood of $E$ with smooth, strictly pseudoconvex boundary.
On $W\setminus E$, one can construct a complete K\"ahler metric $\omega$ with additional good properties, a so-called Saper metric. Here we recall the construction from \cite{ohsawa-isolated-sing}: There exists a nonsingular $m\times m$  matrix $(p_{ij})$ with positive integral coefficients such that
\begin{enumerate}
	\item $\otimes_{i=1}^m \mathcal{O}(-p_{ij}E_i)$ admits a Hermitian metric $h_j$ of positive curvature for all $j$.
	\item Let  $1\leq i_1 < \ldots < i_k\leq m$, $1\leq k\leq m$. Then $\det(p_{i_\alpha i_\beta})_{\alpha,\beta =1}^k\not=0$ whenever $\bigcap_{\alpha=1}^k E_{i_\alpha}\not=\emptyset$
\end{enumerate}
We choose holomorphic sections $s_i$ of $\mathcal{O}(E_i)$ over $\tilde X$ such that $E_i= \lbrace y\in \tilde X\mid s_i(y)=0\rbrace$, and define $\sigma_j$ to be the length of $s_1^{p_{1j}} \cdot\ldots\cdot s_m^{p_{mj}}$ with respect to $h_j$. Then $-\log\log\sigma_j^{-1}$ is a plurisubharmonic function on $W$. We set
$$\omega = -i\pa\opa\big( \sum_{j=1}^m  \log\log\sigma_j^{-1}\big) + i\pa\opa(-\log\delta) \quad\text{on $W\setminus E$,}$$
where $-\delta: \ol W\longrightarrow [0,1]$ is a smooth plurisubharmonic defining function for $W$ which is strictly plurisubharmonic on a neighborhood of $\pa W$.
Then $\omega$ is a complete K\"ahler metric on $W\setminus E$.\\

Moreover, there exists a smooth function $\psi$ on $W\setminus E$ such that $\omega = i\pa\opa\psi$ and $\vert \pa\psi\vert_\omega$ is bounded. As a consequence we get, using standard methods like the twisting trick from Ohsawa--Takegoshi \cite{ohsawa-takegoshi0},
\begin{enumerate}[label=(\alph*)] 
	\item $H^{p,q}_{L^2}(W\setminus E,\omega)=0$ if $p+q\not= n$;
    \item $H^{p,n-p}_{L^2}(W\setminus E,\omega)$ is Hausdorff for every $0\leq p\leq n$;
	\item $H^{k}_{L^2}(W\setminus E,\omega)=0$ if $k\not= n$;
	\item $H^{n}_{L^2}(W\setminus E,\omega)$ is Hausdorff.
\end{enumerate}

\begin{proof}[Sketch of proof]
Even though the above statements are classical, we sketch a proof for the reader's convenience. First we assume that $p+q\geq n+1$. For  $t\in (0,1)$ it then follows from the Bochner--Kodaira--Nakano inequality (see \cite{demailly-book})  that
	\begin{equation} \label{basic}
		\Vert\opa u\Vert^2_{t\psi} + \Vert\opa^\ast_{t\psi}u\Vert^2_{t\psi} \geq t\Vert u\Vert^2_{t\psi}\end{equation}
	for $u\in \mathcal{D}^{p,q}(W\setminus E)$.

Now we substitute $u = v e^{t\psi/2}$. It is not difficult to see that
$$
\vert \opa u\vert^2_\omega  e^{-t\psi} 
 = \vert \opa v + \frac{t}{2} \opa \psi \wedge v \vert^2_\omega  
 \leq 2 \left( \vert\opa v\vert^2_{\omega} + \frac{t^2}{4} \vert \opa\psi\vert^2_{\omega} \vert v\vert^2_{\omega} \right) 
 \leq 2 \vert\opa v\vert^2_{\omega} + \frac{t}{4} \vert v\vert^2_{\omega}
$$
if $t$ is chosen small enough such that $t\vert\opa\psi\vert^2_{\omega}\leq \frac{1}{2}$.
Since $\opa^\ast_{t\psi} = e^{t\psi} \opa^\ast e^{-t\psi}$, we likewise get
\begin{align*}
\vert\opa^\ast_{t\psi} u\vert^2_{\omega} e^{-t\psi} 
 \leq 2\vert\opa^\ast v\vert^2_{\omega} + \frac{t}{4}\vert v\vert^2_{\omega}.
\end{align*}

Together with (\ref{basic}), these two inequalities imply
\begin{equation}  \label{estspace2}
	 \frac{t}{2} \Vert v\Vert^2 \leq \Vert\opa v\Vert^2 + \Vert\opa^\ast v\Vert^2
\end{equation}
for all $v\in\mathcal{D}^{p,q}(W\setminus E)$, hence for all $v\in L^2_{p,q}(W\setminus E,\omega)$ by completeness of the metric $\omega$. This proves (a) for $p+q\geq n+1$. By duality we also obtain
\begin{equation*}  
	 \frac{t}{2} \Vert v\Vert^2 \leq \Vert\opa v\Vert^2 + \Vert\opa^\ast v\Vert^2
\end{equation*} 
for all $v\in L^2_{p,q}(W\setminus E,\omega)$ for $p+q\leq n-1$, which proves (a) and (b). The proof of (c) and (d) is similar.
\end{proof}

So in particular we have that $H^{n-1,1}_{L^2}(W\setminus E,\omega)$ is separated and thus $\Im (\opa: L^2_{n-1,0}(W\setminus E,\omega)\longrightarrow  L^2_{n-1,1}(W\setminus E,\omega))$ can be characterized by moment conditions:
\begin{align*}
\Im (\opa: L^2_{n-1,0}(W\setminus E,\omega)\longrightarrow  L^2_{n-1,1}(W\setminus E,\omega)) = \hspace{2cm}
  \\ \lbrace f\in L^2_{n-1,1}(W\setminus E,\omega)\mid \int_{W\setminus E} f\wedge g =0 \ \forall g\in L^2_{1,n-1}(W\setminus E,\omega) \cap\Ker \opa\rbrace.
\end{align*}

Now let $\chi\in\mathcal{C}^\infty(\tilde X)$ satisfy $\chi\equiv 1$ outside a sufficiently small neighborhood $V$ of $E$ in $W$ and $\chi\equiv 0$ in a smaller neighborhood $U\subset V$ of $E$. We set $f = \opa(\chi\Omega)$.\\

\begin{Claim}  $f\in \Im (\opa: L^2_{n-1,0}(W\setminus E,\omega)\longrightarrow  L^2_{n-1,1}(W\setminus E,\omega)) $.
	\end{Claim}

\begin{proof}
Let $g\in  L^2_{1,n-1}(W\setminus E,\omega) \cap\Ker \opa$ and $h$ its orthogonal projection onto $\mathcal{H}^{1,n-1}(W\setminus E,\omega)$. Then
$$\int_{W\setminus E}f\wedge g = \int_{W\setminus E} f\wedge h = \int_{W\setminus E} d(\chi\Omega)\wedge h = \int_{V\setminus U} d(\chi\Omega)\wedge h$$
Since $h$ is harmonic and $\omega$ is complete and K\"ahler, we have $dh=0$. But then, provided $V$ is a small enough neighborhood of $E$, it follows from \cite{ohsawa-isolated-sing}*{Theorem 15} that there exists $v\in L^2_{n-1}(V^\prime\setminus E,\omega)$ satisfying $dv = h$, where $V^\prime$ is a slightly larger open set containing $\ol V$. Since $h$ is harmonic, thus smooth in $W\setminus E$, it is no loss of generality to assume that $v$ is smooth in $V^\prime\setminus E$ as well. We get
\begin{align*}
\int_{W\setminus E}f\wedge g = \int_{V\setminus U} d(\chi\Omega)\wedge dv = \int_{V\setminus U} d(\chi\Omega\wedge dv)\\
= \int_{\pa V}\chi\Omega\wedge dv = \int_{\pa V}\Omega\wedge dv = \int_{\pa V}d(\Omega\wedge v)=0,
\end{align*}
where we have used the $d$-closedness of $\Omega$ in the last line. Hence $f$ satisfies the moment conditions, so there exists $u\in L^2_{n-1,0}(W\setminus E,\omega)$ satisfying $\opa u = f$ in $W\setminus E$.
\end{proof}

We next claim that we can improve the integrability of the solution $u$ of $\opa u = f$ near $E$.

\begin{Claim}
	There exists $\eta >0$ and $u\in  L^2_{n-1,0}(W\setminus E,\omega, \eta\varphi )$ satisfying $\opa u = f$, where $\varphi = -\sum_{j=1}^m\log\log\sigma_j^{-1}$.
\end{Claim}

\begin{proof}
$\varphi$ is a plurisubharmonic function, therefore also the functions
$$\varphi_\varepsilon = \max_\varepsilon (-\frac{1}{\varepsilon},\varphi),$$
where $\max_\varepsilon$ is a regularized maximum function, are plurisubharmonic. That is, for each $\varepsilon > 0$, we let $\lambda_\varepsilon$ be  a fixed non negative real smooth function on $\R$ such that, for all $x\in\R$, $\lambda_\varepsilon (x) = \lambda_\varepsilon (-x)$, $\vert x\vert \leq \lambda_\varepsilon (x)\leq \vert x\vert + \varepsilon$, $\vert\lambda_\varepsilon^\prime(x)\vert \leq 1$, $\lambda_\varepsilon^{\prime\prime}(x)\geq 0$ and  $\lambda_\varepsilon (x) = \vert x\vert$ if $\vert x\vert\geq \frac{\varepsilon}{2}$. We moreover assume that $\lambda_\varepsilon^\prime (x) > 0$ if $x> 0$ and $\lambda_\varepsilon^\prime (x) < 0$ if $x < 0$ and set $\max_\varepsilon (t,s) = \frac{t+s}{2} + \lambda_\varepsilon (\frac{t-s}{2})$ for $t,s\in\R$.\\

Using the twisting trick of Ohsawa--Takegoshi again, one can show  the following a priori estimate: For all $u\in L^2_{n-1,0}(W\setminus E,\omega)$ we have
\begin{equation} \label{apriori}
	\Vert u\Vert^2_{\eta\varphi_\varepsilon}\leq C \Vert \opa u\Vert^2_{\eta\varphi_\varepsilon}
\end{equation}
for some $\eta > 0$, where $C$ is independent of $\varepsilon$.\\

Now let $u\in L^2_{n-1,0}(W\setminus E,\omega)$ be any solution satisfying $\opa u = f$. Then, according to (\ref{apriori}), we have
$$\Vert u\Vert^2_{\eta\varphi_\varepsilon}\leq C \Vert f\Vert^2_{\eta\varphi_\varepsilon}$$
for any $\varepsilon > 0$, hence we get $\Vert u\Vert_{\eta\varphi} < \infty$. 
\end{proof}

We denote by $\tilde u$ the extension of $u$ by zero outside $W\setminus E$. We will now prove that we have $\opa\tilde u= f$ in $\tilde X$. This will follow from the next two claims.\\

\begin{Claim} $\opa \tilde u =0$ in $U$ in the sense of distributions.
	\end{Claim}

\begin{proof}
We begin by estimating $\omega$ near $E$. So let $(v,w)= (v_1,\ldots ,v_k, w_1,\ldots, w_{n-k})$ denote the coordinates on a neighborhood $U_k$ of  a $k$-ple point of $E$ such that $v_1\cdot\ldots\cdot v_k=0$ is a local defining equation of $E$. Then, on $U_k$, we have
\begin{equation} \label{metric}
	\omega \sim \sum_{i=1}^k \frac{dv_i\wedge d\ol v_i}{\vert v_i\vert^2\log^2\vert v_1\cdot\ldots\cdot v_k\vert^{-1}} \\ + \frac{1}{\log\vert v_1\cdot\ldots\cdot v_k\vert^{-1}}
	(\sum_{i=1}^k dv_i\wedge d\ol v_i + \sum_{j=1}^{n-k} dw_j\wedge d\ol w_j)
	\end{equation}
	
We fix an arbitrary smooth Hermitian metric $\omega_o$ on $\tilde X$. Then, since $u$ is of bidegree $(n-1,0)$, we get from (\ref{metric}) that
$$\int_{U_k \setminus E} \vert u\vert^2_\omega dV_\omega \gtrsim \int_{U_k \setminus E} \vert u\vert^2_{\omega_o} \frac{1}{\log\vert v_1\cdot\ldots\cdot v_k\vert^{-1}}  dV_{\omega_o},$$
hence 
$$\big(\int_{U_k \setminus E} \vert u\vert_{\omega_o} dV_{\omega_o}\big)^2\leq \int_{U_k \setminus E} \vert u\vert^2_{\omega_o} \frac{1}{\log\vert v_1\cdot\ldots\cdot v_k\vert^{-1}}  dV_{\omega_o} \cdot 
\int_{U_k \setminus E} \log\vert v_1\cdot\ldots\cdot v_k\vert^{-1}dV_{\omega_o} < \infty.$$
Thus $u\in L^1_{n-1,0}(W\setminus E, \omega_o)$.

Now, to prove the claim, we have to show that
\begin{equation}  
	\int_{U} \tilde{u} \wedge \opa g = 0
\end{equation}
for $g\in \mathcal{D}^{1,n-1}(U)$. Using a partition of unity, we may assume that $\mathrm{supp}\ g$ is contained in $U_k$.
Note that it will be essential in the following argument that $u\in L^2_{n-1,0}(W\setminus E,\omega, \eta\varphi )$.

Let $\chi\in\mathcal{C}^\infty(\R,\R)$ be a function such that $\chi(t)=0$ for $t \leq \frac{1}{2}$ and $\chi(t)=1$ for $t \geq 1$. Set $\chi_j = \chi(j(\log\vert v_1\cdot\ldots\cdot v_k\vert^{-1})^{-1})$. Since $u\in L^1_{n-1,0}(W \setminus E,\omega_o)$, we have
$$\int_U \tilde u\wedge \opa g = \lim_{j\rightarrow \infty}\int_{U_k} u\wedge\chi_j\opa g= 
\lim_{j\rightarrow \infty}\int_{U_k}\tilde u\wedge\opa(\chi_j g) -\lim_{j\rightarrow \infty}\int_{U_k}\tilde u\wedge\opa\chi_j\wedge g.$$
From the Cauchy--Schwarz inequality we obtain
\begin{equation} \label{a}
	\vert\int_{U_k} u\wedge\opa\chi_j\wedge g\vert^2 \leq \int_{U_{kj}} \vert u\vert^2_\omega e^{-\eta\varphi} dV_\omega \cdot \int_{U_k}\vert\opa\chi_j\wedge g\vert^2_\omega e^{\eta\varphi} dV_\omega,
	\end{equation}
where $U_{kj}= \lbrace z\in U_k\mid 0\leq \vert v_1\cdot \ldots\cdot v_k\vert\leq \exp(-j)\rbrace$.

Now it follows from (\ref{metric}) that
$$\int_{U_k}\vert\opa\chi_j\wedge g\vert^2_\omega e^{\eta\varphi} dV_\omega\lesssim \int_{U_k}\vert g\vert^2_\omega e^{\eta\varphi} dV_\omega.$$
Moreover, since $g$ is an $n$-form,
$$\vert g\vert^2_\omega \lesssim (\log\vert v_1\cdot\ldots\cdot v_k\vert^{-1})^n$$
and 
$$dV_\omega \sim\vert v_1\cdot\ldots\cdot v_k\vert^{-2}(\log\vert v_1\cdot\ldots\cdot v_k\vert^{-1})^{-n-k}$$
if $\vert v\vert\leq \frac{1}{2}$. Therefore, for some $\tau > 0$

$$\vert g\vert^2_\omega e^{\eta\varphi}dV_\omega \lesssim \vert v_1\cdot\ldots\cdot v_k\vert^{-2} (\log\vert v_1\cdot\ldots\cdot v_k\vert^{-1})^{-k-\tau}$$
and thus
\begin{align*}
\int_{U_k}\vert g\vert^2_\omega e^{\eta\varphi} dV_\omega &\lesssim
\int_0^{1/2}\ldots\int_0^{1/2} (t_1\cdot\ldots\cdot t_k)^{-1}\times (\log(t_1\cdot\ldots\cdot t_k)^{-1})^{-k-\tau}dt_1\cdots dt_k\\
&\lesssim ( \int_0^{1/2}t^{-1}(\log t^{-1})^{-1-\tau/k})dt)^k < \infty.
\end{align*}
But this implies that $\lim_{j\rightarrow \infty}\int_{U_k}\tilde u\wedge\opa\chi_j\wedge g =0$, hence
$$\int_U \tilde u\wedge \opa g = \lim_{j\rightarrow \infty}\int_{U_k}\tilde u\wedge\opa(\chi_j g) =0,$$
since $\tilde u$ is $\opa$-closed in $U \setminus E$.
\end{proof}

\begin{Claim}
We also have $\opa\tilde u = 0$ outside $V$.
\end{Claim}

\begin{proof}
We have to show that $$0=\int_{\tilde X\setminus V}\tilde u \wedge\opa\alpha = \int_{W\setminus V} u\wedge\opa\alpha$$ for all $\alpha\in\mathcal{D}^{1,n-1}(\tilde X\setminus V)$.\\

To see this, we again fix an arbitrary smooth Hermitian metric $\omega_o$ on $\tilde X$. Since $i\pa\opa(-\log\delta) = i\frac{\pa\opa(-\delta)}{\delta} + i\pa\log\delta\wedge\opa\log\delta$, we may conclude that we have \\
$\int_{W\setminus U}\vert u\vert^2_{\omega_o}\delta^{-1}dV_{\omega_o} < \infty$. \\

	Now let $\chi\in\mathcal{C}^\infty(\R,\R)$ be a function such that $\chi(t)=0$ for $t \leq \frac{1}{2}$ and $\chi(t)=1$ for $t \geq 1$. Set $\chi_j = \chi(j\delta)\in\mathcal{D}(W)$. 
	Then $\chi_j \alpha \in\mathcal{D}^{1,n-1}(W\setminus V)$, and since $\opa u =0$ in $W\setminus V$, we therefore have
	$$0=\int_{W\setminus V} u\wedge \opa(\chi_j \alpha) = \int_{W\setminus V} u\wedge\opa \chi_j \wedge\alpha +  \int_{W\setminus V} u\wedge \chi_j \opa \alpha.$$
	As $u$ has $L^2$ coefficients on $W$, the integral of $u\wedge\chi_j \opa\alpha$   converges  to the integral of $u\wedge \opa\alpha$  as $j$ tends to infinity.  
	The remaining term can be estimated as follows: using the Cauchy--Schwarz inequality we have
	\begin{equation*}  
		\left\vert\int_{W} u\wedge \opa\chi_j \wedge \alpha\right\vert^2 \leq
		\sup_{W}\vert \alpha\vert^2_{\omega_o} \int_{\lbrace 0<\delta \leq \frac{1}{j}\rbrace} \vert u\vert^2_{\omega_o}\delta^{-1}dV_{\omega_o} \cdot \int_{W}\vert\opa\chi_j\vert^2_{\omega_o}\delta dV_{\omega_o}.
	\end{equation*}
	Note that the last integral is bounded.
	Since $\int_{W\setminus U}\vert u\vert^2_{\omega_o}\delta^{-1}dV_{\omega_o} < \infty$, the integral $\int_{\lbrace 0 <\delta \leq \frac{1}{j}\rbrace} \vert u\vert^2_{\omega_o}\delta^{-1}dV_{\omega_o}$ converges to $0$ when $j$ tends to infinity. 

We have thus proved that $\int_{W\setminus V} u\wedge\opa\alpha =0$, hence $\opa\tilde u =0$ outside $V$.
\end{proof}

Combining the above two claims, we have thus proved that  $\opa\tilde u =f$ in $\tilde{X}$. Setting $\tilde\Omega = \chi\Omega- \tilde f$ we then obtain a holomorphic extension of $\Omega$ to $\tilde X$.\\
	
Now we are ready to prove the desired extension result on the original manifold $X$. Namely let $\Omega$ be a $d$-closed holomorphic $(n-1)$ form on $X\setminus A$. We  choose  a smooth $(n-1)$-form $\hat\Omega$ on $X$ such that $\hat\Omega = \Omega$ outside a compact set of $X$  and set $f = \opa\hat\Omega\in \mathcal{D}^{n-1,1}(X)$. Then, by standard $\opa$-theory, $\Omega$ admits a holomorphic extension to $X$ if the equation $\opa u = f$ admits a smooth solution $u$ with compact support in $X$. Since $X$ is not compact, it follows from Serre duality that the Dolbeault cohomology group with compact support in $X$, $H_c^{n-1,1}(X)$, is separated. Therefore a sufficient condition for the existence of a compactly supported solution $u$ of $\opa u = f$ is that
$$\int_X f\wedge g = 0 \qquad \forall g\in\mathcal{C}^\infty_{1,n-1}(X)\cap\mathrm{Ker}(\opa).$$
Choose a relatively compact domain $D$ with smooth boundary  in $X$ such that $\mathrm{supp} f\subset D$. Then for $ g\in\mathcal{C}^\infty_{1,n-1}(X)\cap\mathrm{Ker}(\opa)$ we get
\begin{align*}
    \int_X f\wedge g & =  \int_D \opa\hat\Omega\wedge g 
    =  \int_D d(\hat\Omega\wedge g) \\
    & = \int_{\pa D} \Omega\wedge g = \int_{\pi^{-1}(\pa D)} \pi^\ast(\Omega) \wedge\pi^\ast (g)
\end{align*}
Now $\pi^{-1}(\pa D)$ bounds a domain $\tilde D$ in $\tilde X$, and we have proved before that $\pi^\ast(\Omega)$ admits a holomorphic extension $\tilde \Omega$ over the exceptional divisor $E$. But then
$$\int_{\pi^{-1}(\pa D)} \pi^\ast(\Omega) \wedge\pi^\ast (g) = \int_{\pa\tilde D} \tilde\Omega\wedge\pi^\ast(g) = \int_{\tilde D}d(\tilde\Omega\wedge\pi^\ast(g)) = \int_{\tilde D}\opa(\tilde\Omega\wedge\pi^\ast(g))=0,$$
which concludes the proof.
\end{proof}

\section{Applications}
\label{sect:applications}

\subsection{Another proof of Brunella's conjecture}
In this subsection, we combine the residue formula (Theorem \ref{thm:residue}) with the extension result (Theorem \ref{thm:extension}) to 
give another proof of the following theorem, which was conjectured by Brunella \cite{brunella1} and recently confirmed in \cite{adachi-brinkschulte}:

\begin{Theorem}[\cite{adachi-brinkschulte}] \label{thm:brunella}
Let $X$ be a compact complex manifold of dimension $n \geq 3$.
Let $\mathcal{F}$ be a codimension one holomorphic foliation on $X$ with ample normal bundle $N_\mathcal{F}$. 
Then every leaf of $\mathcal{F}$ accumulates to $\Sing(\mathcal{F})$.
\end{Theorem}

The main point of the proof in \cite{adachi-brinkschulte} was a localization argument for the first Chern class $c_1(N_\mathcal{F})$ of the normal bundle to the holomorphic foliation $\mathcal{F}$.
Due to the lack of a suitable residue formula, 
the localization was proved in \cite{adachi-brinkschulte} via the first Atiyah class $a_1(N_\mathcal{F})$.
We can give more direct approach to this localization thanks to Theorems \ref{thm:residue} and \ref{thm:extension}.

\begin{proof}
The proof is by contradiction. Assume that there exists a leaf $\mathcal{L}$ whose closure 
$M := \overline{\mathcal{L}}$ does not intersect  $S := \Sing(\mathcal{F})$.
We may assume that $X$ is connected by replacing it with the connected component containing $M$ if necessary.
From Brunella's convexity result \cite{brunella2}*{Proposition 3.1}, 
$X \setminus M$ is strongly pseudoconvex.
Denote by $A$ the maximal compact analytic set of $X \setminus M$,
and by $S^*$ the codimension two part of $S$. 
Note that $S^* \subset A$ since $\dim X \geq 3$.
From \cite{adachi-brinkschulte}*{Section 3.1}, we have a holomorphic connection $\nabla_{\rm hol}$ of $N_\mathcal{F}|_{X \setminus A}$.

Note that $X$ is projective since $X$ possesses an ample line bundle $N_\mathcal{F}$.
Thanks to Hironaka's desingularization (see Section \ref{sect:mod}),  we may assume that there exists a proper holomorphic mapping $\pi: \tilde X \longrightarrow X$ such that  $D=\pi^{-1}(A)$ is a simple normal crossing divisor in a projective manifold $\tilde{X}$.
The pull-back bundle $L := \pi^*N_{\mathcal{F}}$ is a holomorphic line bundle admitting a holomorphic connection $\pi^*\nabla_{\rm hol}$ over $\tilde{X} \setminus D$. 
Denote by $\Omega$ the curvature form of $\pi^*\nabla_{\rm hol}$, which is a closed holomorphic 2-form on $\tilde{X} \setminus D$.
From Theorem \ref{thm:extension} and Remark \ref{rem:extension}, we see that $\Omega$ extends over $D$ since $\dim \tilde{X} = \dim X \geq 3$. 
Then, from Theorem \ref{thm:residue}, $\pi^*\nabla_{\rm hol}$ is, in fact, integrable and the first Chern class of $L$ decomposes as
\[
c_1(L) = - \sum_{\nu=1}^N \operatorname{Res}_{D_\nu}(\pi^*{\nabla_{\rm hol}}) [D_\nu] \in H^{1,1}(\tilde{X})
\]
where $D = \sum_{\nu=1}^N D_\nu$ is the irreducible decomposition.

The rest of the proof is exactly as in \cites{canales,adachi-biard}. 
We take an arbitrary smooth Hermitian metric $h_L$ of $L$, and Hermitian metrics $h_\nu$ of the line bundle 
defined by $D_\nu$ so that the curvature of $h_\nu$ has support in $\tilde{X} \setminus \pi^{-1}(M)$.
Then, from the $\partial\overline{\partial}$-lemma, there exists a smooth real function $\psi$ on $Y$ such that 
\[
i\Theta_{h_L} = - \sum_{\nu=1}^N \Re (\operatorname{Res}_{D_\nu}({\nabla})) i\Theta_{h_\nu} + i\partial\overline{\partial} \psi.
\]
Then, $h_0 := h_L e^{-\psi}$ gives a smooth flat Hermitian metric of $L$ over an open neighborhood of $\pi^{-1}(M)$.

Since we assume that $N_{\mathcal{F}}$ is ample, there exists a Hermitian metric $h_+$ of $L$ over $\pi^{-1}(M)$ with (leafwise) positive curvature. 
Then, $\varphi := -\log h_+/h_0 \colon \pi^{-1}(M) \to \R$ is a continuous function which is leafwise strictly plurisubharmonic.
But this violates the maximum principle for strictly plurisubharmonic functions on the leaf where $\varphi$ takes its maximum.
This contradiction completes the proof. 
\end{proof}

\subsection{Another proof for a non-existence theorem of certain Levi flat hypersurfaces}

In this subsection, we remark a simplified proof for the following theorem:

\begin{Theorem}[\cites{canales,adachi-biard}] \label{thm:affine}
Let $X$ be a compact K\"ahler surface and $M$ a real analytic closed Levi flat hypersurface whose complement is 1-convex. 
Then the Levi foliation of $M$ cannot be transversely affine. 
\end{Theorem}

A weaker form of this result was first proved for projective surfaces by Canales \cite{canales} using the residue formula of Cousin and Pereira \cite{cousin-pereira}*{Proposition 2.2}. 
Since the residue formula of Cousin and Pereira was known to be valid for projective manifolds only, we employed in \cite{adachi-biard} the residue formula of Ohtsuki \cite{ohtsuki}, requiring a logarithmic connection, and a theorem of Deligne \cite{deligne}*{Proposition II.5.4} to prove this theorem for compact K\"ahler surfaces.
Now this technical point can be simplified thanks to Theorem \ref{thm:residue_integrable}.

\begin{proof}
The proof is by contradiction. Assume that there exists a real analytic closed Levi flat hypersurface $M$ in $X$ such that $X \setminus M$ is 1-convex and the Levi foliation of $M$ is transversely affine. 

 Since $X\setminus M$ is assumed to be 1-convex, the Remmert reduction $\phi$ permits us to view $X\setminus M$ as a modification of a 2-dimensional normal Stein space $Y$ (see Section \ref{sect:mod}). The maximal compact analytic set $A\subset X\setminus M$ is then a modification of a finite set $B \subset Y$.
Using this fact, we can extend the Levi foliation of $M$ to a possibly singular holomorphic foliation $\mathcal{F}$ on $X \setminus A$ by a Bochner--Hartogs type extension theorem (see \cite{canales}*{\S 5} or \cite{Ivashkovich}*{\S 4.3}).

We claim that $\mathcal{F}$ can be extended as a holomorphic foliation all over $X$: 
Since $Y$ is a normal Stein space, Cartan's Theorem A guarantees the existence of two global sections $\omega_1, \omega_2$ of the cotangent sheaf $\Omega^1_Y$ such that $\omega_1\wedge \omega_2$ does not vanish identically (see for instance \cite{GunningRossi}*{p.\ 243}).
If $\omega_2$ defines $\mathcal{F}$ over $Y \setminus B$, the desired extension of $\mathcal{F}$ over $X$ is obtained by $\phi^*\omega_2$. 
If not, by pointwise linear algebra, we can find a meromorphic function $h$ on $Y \setminus B$ such that the meromorphic 1-form $\omega_1+h\omega_2$ defines the foliation $\mathcal{F}$ over $Y \setminus B$:
We work locally on a holomorphic chart with coordinate $(z,w)$ where we can write
\[
\omega_1 = \alpha_1 dz + \alpha_2 dw, \quad \omega_2 = \beta_1 dz + \beta_2 dw
\]
with holomorphic functions $\alpha_1, \alpha_2, \beta_1, \beta_2$ and we take a defining holomorphic 1-form $f dz + g dw$ for $\mathcal{F}$. Then one can find the unique meromorphic solution
\[
h = \frac{g\alpha_1 - f\alpha_2}{g \beta_1 - f \beta_2}.
\]
Since $B$ has codimension $\geq 2$, we can apply a version of Levi's extension Theorem (see, e.g., \cite{Ivashkovich}*{Corollary 1.5}) and extend $h$ into a meromorphic function $\tilde{h}$ to all of $Y$.
The meromorphic 1-form $\phi^*(\omega_1+\tilde{h}\omega_2)$ defines the desired extension of $\mathcal{F}$ over $X$ (see Sections 2.2 and 2.3). 

Now we use Hironaka's desingularization $\pi \colon \tilde{X} \to X$ (see Section \ref{sect:mod}), where $\tilde{X}$ is a compact K\"ahler surface and $\pi^{-1}(A)=D$ is a normal crossing divisor.
Using this map $\pi$, we pull back the extension of $\mathcal{F}$ over $X$ to a foliation $\tilde{\mathcal{F}}$ over $\tilde{X}$.
We identify $M$ with $\pi^{-1}(M)\subset \tilde{X}$.
From the real analyticity of $M$, $\tilde{\mathcal{F}}$ is transversely affine in a small neighborhood $W$ of $M$ (see \cite{canales}*{Proposition 6.3}) and its normal bundle $N_{{\mathcal{F}}}|_W$ admits an integrable connection (see \cite{adachi-biard}*{Lemma 12}). From a Bochner--Hartogs type extension theorem, this connection extends to an integrable connection of $N_{\tilde{\mathcal{F}}}|_{\tilde{X} \setminus D}$ (see \cite{adachi-biard}*{p.379}). 
Then, by Theorem \ref{thm:residue_integrable}, we have the decomposition of the first Chern class of $N_{\tilde{\mathcal F}}$,
\begin{eqnarray}\label{Chernclass}
c_1(N_{\tilde{\mathcal{F}}}) = - \sum_{\nu=1}^N \operatorname{Res}_{D_\nu}(\nabla) [D_\nu] \in H^{1,1}(\tilde{X})
\end{eqnarray}
where $D = \sum_{\nu=1}^N D_\nu$ is the irreducible decomposition of the divisor.

The rest of the proof is exactly as in \cites{adachi-biard}. We give a sketch for the reader's convenience. 
Fix an arbitrary smooth Hermitian metric $h_0$ of $N_{\tilde{\mathcal{F}}}$. For each $\nu$, we take smooth Hermitian metrics $h_\nu$ of the line bundle associated with $D_\nu$ so that the support of the curvature form $\Theta_{h_\nu}$ does not intersect with the neighborhood $W$ of $M$. 
From \eqref{Chernclass}, since $\tilde{X}$ is K\"ahler, we may use the $\pa\opa$-lemma to obtain a smooth real function $\psi$ on $\tilde{X}$  such that 
$$
\frac{i}{2\pi}\Theta_{h_0} = - \sum_{\nu=1}^N  \Re(\operatorname{Res}_{D_\nu} \nabla) \frac{i}{2\pi}\Theta_{h_\nu} + i\pa\opa \psi.
$$
Then $h := h_0 e^{2\pi\psi}$ is a smooth Hermitian metric of $N_{\tilde{\mathcal{F}}}$ 
whose curvature form
vanishes away from the supports of $\Theta_{h_\nu}$. 
Hence, the restriction of $h$ yields a flat Hermitian metric of $N_{\mathcal{F}}$ over $W$.

In each foliated chart $(U_\alpha, (z_\alpha, w_\alpha=t_\alpha+iu_\alpha))$ of $\mathcal{F}|_W$, we define $h_\alpha= h(\frac{\pa}{\pa w_\alpha}, \frac{\pa}{\pa w_\alpha})$. Because $\mathcal{F}$ is transversely affine on $M$, we have the well-defined smooth 1-form $\omega := \sqrt{h_\alpha(z_\alpha, t_\alpha)} dt_{\alpha}$ on $M$ defining the Levi foliation and such that
$$d\omega = \left(\frac{\pa \sqrt{h_\alpha(z_\alpha, t_\alpha)}}{\pa z_\alpha} dz_\alpha +
\frac{\pa \sqrt{h_\alpha(z_\alpha, t_\alpha)}}{\pa \ol{z}_\alpha} d\ol{z}_\alpha\right) \wedge dt_\alpha
= (\eta + \ol{\eta}) \wedge \omega,
$$
where $$\eta := \frac{\pa (\log \sqrt{h_\alpha(z_\alpha, t_\alpha)})}{\pa z_\alpha} dz_\alpha.$$
Because $\Theta_h=0$, applying Stokes' Theorem on $d\omega$, we deduce that $\omega$ is a closed 1-form on $M$. Hence, by its definition, $h_\alpha$ is independent of $w_\alpha$ and then $\rho(z_\alpha, w_\alpha) := \sqrt{h_\alpha(t_\alpha)} u_\alpha$ is a defining function of $M$ in $U_\alpha$. From $\rho$ we can construct displacements of $M$ in $X\setminus M$ that are smooth closed Levi flat hypersurfaces, $\rho^
{-1}(\varepsilon)$ for $|\varepsilon| \ll 1$.
This existence of smooth closed Levi flat hypersurfaces contradicts the 1-convexity of $X \setminus M$ due to the maximum principle for strictly plurisubharmonic function, and completes the proof.
\end{proof}

\subsection{Remark on minimal sets of foliations on complex surfaces}
By combining the arguments for Theorems \ref{thm:brunella} and \ref{thm:affine}, we can deduce the following result.
\begin{Theorem}
Let $X$ be a compact complex surface, and 
$\mathcal{F}$ a codimension one holomorphic foliation on $X$ with ample normal bundle $N_\mathcal{F}$. 
If $\mathcal{F}$ admits a non-trivial minimal set $\mathcal{M}$, then $\mathcal{F}$ is not transversely affine near $\mathcal{M}$.
\end{Theorem}

\begin{proof}[Sketch of proof]
Again, the proof is by contradiction. Assume that there exists a non-trivial minimal set $\mathcal{M}$ around which $\mathcal{F}$ is transversely affine. 
Since $\mathcal{F}$ is transversely affine, we have an integrable connection $\nabla$ of $N_\mathcal{F}$ around $\mathcal{M}$. 
From Brunella's convexity result \cite{brunella2}*{Proposition 3.1}, $X \setminus \mathcal{M}$ is strongly pseudoconvex.
We can see that, as in the proof of Theorem \ref{thm:affine}, $\nabla$ extends to an integrable connection over $X \setminus A$
where $A$ is the maximal compact analytic set of $X \setminus M$.
We can apply the arguments in the proof of Theorem \ref{thm:brunella} to obtain a continuous function on $\mathcal{M}$ which is leafwise strictly plurisubharmonic. This yields a contradiction. 
\end{proof}

\begin{Remark}
Minimal sets near which the foliation is transversely affine can exist when the normal bundle $N_\mathcal{F}$ is not ample.
There is a real-analytic Levi flat hypersurface with transversely affine foliation in hyperbolic Inoue surfaces. See, e.g., \cite{adachi-biard}*{Example 14} for its detail.
\end{Remark}

\end{document}